\documentclass[11pt]{amsart}
\usepackage{mypackage}
\usepackage{myformat}
\usepackage{hyperref}
\hypersetup{colorlinks=true,linkcolor=cyan}

\newcommand{\Id}		{\mathcal I}

\newcommand {\Mon}	{\mathfrak M}     
     
\newcommand {\Proj}	{\mathfrak P}     
\newcommand {\SProj}	{\mathfrak Q}     
\newcommand{\geId}	{\ge_{_\Id}}

\newcommand{\gez}	{\ge_{_{\{0\}}}}

\newcommand{\na}	{\mathfrak n}

\newcommand{\VM}{\mathscr V(\Mon)}

\begin{document} 
\title[Positively ordered Monoids]
{The Projection Problem in Commutative, Positively Ordered Monoids}
\mybic
\date \today 


\begin{abstract} 
We examine the problem of projecting subsets of a 
commutative, positively ordered monoid into an 
$o$-ideal. We prove that to this end one may restrict 
to a sufficient subset, for whose cardinality we provide 
an explicit upper bound. Several applications to set 
functions, vector lattices and other more explicit 
structures are provided. 
\end{abstract}

\maketitle

\section{Introduction.}

Several problems in analysis are greatly simplified
by the possibility of reducing the cardinality of the 
set under scrutiny from arbitrary to finite -- or at least 
countable -- which follows from compactness or 
separability. In this paper we explore the possibility 
of a similar simplification arising from a notion of a 
purely set theoretic nature, $\kappa$-ideals, a concept 
originally introduced by Tarski \cite{tarski_45} that we 
adapt to the study of commutative, positively ordered 
monoids (or semigroups). In the context of this mathematical 
structure we define, section \ref{sec ideal}, ideals and 
projections and investigate, in section \ref{sec projection}, 
the projection problem, that is the problem of projecting 
a subset of a positively ordered monoid into a given ideal. 
We show in Theorem \ref{th main} that a set can be projected 
on an ideal if and only if the same is true for any of its subsets 
with cardinality less than some explicit bound -- often just 
countable subsets. The proof is elementary and exploits 
some properties of cardinal numbers. Given the simple 
mathematical structure of positively ordered monoids, 
our result is quite general and abstract, although it has 
almost immediate applications to lattices, Boolean algebras 
and to families of set functions. The main applications of
these results are developed in section \ref{sec FV} where
we introduce the class $\VM$ of functions of finite variation
defined on a p.o. monoid $\Mon$. As an application of 
Theorem \ref{th main}, we obtain in Theorem \ref{th VM} a 
necessary and sufficient condition for a subset 
$V\subseteq\VM$ to admit a strictly positive element.
This problem is connected to Maharam problem in the
theory of additive functions on Boolean algebras.

To simplify definitions, throughout the paper we assume 
the commutative property without explicit mentioning, so 
that a monoid or semigroup is always meant to be commutative.

For the rest of the paper, and without further mention,
$\Mon$ will be a positively ordered (p.o.) monoid, as 
defined by Clifford \cite[p. 308]{clifford_58}. That is, 
$\Mon$ is a monoid (written multiplicatively and with 
$1$ designating its unit) endowed with a partial order 
$\ge$ that satisfies
\begin{subequations}
\label{order}
\begin{equation}
\label{order 1}
1\ge m
\qquad
m\in\Mon
\qtext{and}
\end{equation}
\begin{equation}
\label{order fg}
m\ge m'
\qtext{implies}
mh\ge m'h
\qquad
m,m',h\in\Mon.
\end{equation}
\end{subequations}
Every monoid is a p.o. monoid if we write $m\ge n$ 
whenever $m$ divides $n$%
\footnote{
A p.o. semigroup $\mathfrak S$ is defined likewise, 
but replacing \eqref{order 1} with the condition:
$m\ge mn$ for all $m,n\in\mathfrak S$. Given that 
each p.o. semigroup may be embedded into a p.o. 
monoid, we shall mainly focus on the latter structure.}
. 

Our results become significantly simpler if we assume, 
with no loss of generality, the existence of a least 
element $0\in\Mon$. Two elements $m,n\in\Mon$ 
are disjoint if $mn=0$ and a set $M\subseteq\Mon$ is 
mutually disjoint if $0\notin M$ and $mn=0$ for all 
distinct pairs $m,n\in M$. When $M\subseteq\Mon$ 
we also use the lattice notation
\begin{equation}
\label{T perp}
M^\perp
	=
\{n\in\Mon:nm=0\text{ for all }m\in M\}.
\end{equation}
At times we shall require in addition that $\Mon$ 
is {\it normal}, i.e. that $0$ is the only nilpotent 
element of $\Mon$, or even that $\Mon$ is 
idempotent (more precisely, that all 
elements of $\Mon$ are idempotent%
\footnote{
Idempotent semigroups are often referred to as 
{\it semilattices}. See Birkhoff \cite[p. 9]{birkhoff}, 
Leader \cite{leader} or Blyth \cite[p. 19]{blyth}
}). 

Many well known mathematical structures, such as
lattices and Boolean algebras, are examples of a 
p.o. monoid or semigroup. If 
$\Mon_\alpha$ is a p.o. monoid for each  
$\alpha\in\mathfrak A$, the product p.o. monoid 
is the set
$\bigtimes_{\alpha\in\mathfrak A}\Mon_\alpha$ 
with order and composition defined coordinatewise. 
The space $\Fun{Z,\Mon}$ of all functions defined 
on some arbitrary set $Z$ and with values in a p.o. 
monoid $\Mon$ is then a product p.o. monoid. 
If $X$ is an {\it arbitrary} set and each $x\in X$ 
is identified, via the evaluation map, with a
function $\hat x\in\Fun{Z,[0,1]}$ where 
$Z=\Fun{X,[0,1]}$, we obtain the {\it abstract
p.o. monoid} $\hat X=\{\hat x:x\in X\}$
associated with $X$.

We denote by $\card A$ the cardinality of a set 
$A$ and refer to $A$ as an $\na$-set if $\na$ is 
a cardinal number and $\na>\card A$. If 
$F\subseteq\Fun{X,Y}$, the image of $A\subseteq X$ 
under $f\in F$ is written as $f[A]$ and we let 
$F[A]=\bigcup_{f\in F}f[A]$.

\section{Preliminary notions: ideals and projections.}
\label{sec  ideal}

Given the interaction between algebraic and order 
properties, several concepts, including ideals and 
projections, may be given distinct definitions 
depending if considered in algebraic or in order
terms. This section is of limited mathematical
content, provides some rigorous definitions and 
proves some basic facts.

A monoid ideal (or simply an ideal, for short) in $\Mon$ 
is a subset $\Id\subseteq\Mon$ such that%
\footnote{
Our definition corresponds to that of a semigroup ideal, 
see e.g. Anderson and Johnson \cite{anderson_johnson} 
or Rees \cite{rees}. We adopt the convention that the 
empty set is an ideal.
}
\begin{equation}
\label{ideal}
m\in\Id
\qtext{and}
n\in\Mon
\qtext{imply}
mn\in\Id;
\end{equation}
an order ideal ($o$-ideal) is a subset 
$\mathcal J\subseteq\Mon$
satisfying the more restrictive condition
\begin{equation}
\label{$o$-ideal}
m\in\mathcal J,\ 
n\in\Mon
\qtext{and}
m\ge n
\qtext{imply}
n\in\mathcal J.
\end{equation}

Given any monoid $\Mon$, even without a partial order, 
an ideal $\Id\subseteq\Mon$ induces a reflexive 
and transitive binary relation $\geId$ defined by 
\begin{equation}
\label{quotient}
m\geId n
\qtext{whenever}
mh\in\Id
\qtext{implies}
nh\in\Id
\qquad
h\in\Mon.
\end{equation}
The associated equivalence relation $\sim_{_\Id}$ 
(i.e. $m\sim_{_\Id}n$ if $m\geId n$ and $n\geId m$),
generates a quotient monoid $\Mon/\sim_{_\Id}$,
which we write more simply as $\Mon/\Id$. Defining 
multiplication of equivalence classes in the usual 
way and letting $m/\Id\ge n/\Id$ if and only if 
$m\geId n$, we obtain a p.o. monoid%
\footnote{
The factor p.o. monoid $\Mon/\Id$ should not be
confused with other quotients, such as the factor 
semigroup $\Mon-\Id$ defined by Rees \cite[p. 389]{rees}. 
The latter need not posses an order structure and is 
induced by the equivalence relation $m=n$ or $m,n\in\Id$. 
Rees congruence implies $\sim_{_\Id}$ but the 
converse need not be true.
}. 
The canonical map $\Mon\to\Mon/\Id$ is a 
homomorphism of monoids but, if $\Mon$ is a
p.o. monoid, it preserves order if and only if $\Id$ 
is an $o$-ideal. If $ M\subseteq\Mon$ we use the 
standard notation 
\begin{equation}
 M/\Id
	=
\{m/\Id:m\in M\}.
\end{equation}
If $m/\Id$ and $n/\Id$ are disjoint we say that $m$ 
and $n$ are $\Id$-disjoint and this is equivalent to 
$mn\in\Id$. A set $ M\subseteq\Mon$ is mutually 
$\Id$-disjoint if $M/\Id$ is mutually disjoint. 

$o$-ideals have special importance. Each set
$A\subseteq\Mon$ generates a corresponding 
$o$-ideal defined as
\begin{equation}
\label{I(A)}
I(A)
	=
\bigcup_{a\in A}\{m\in\Mon:m\le a\}%
\footnote{
We prefer $I(m)$ to $I(\{m\})$, when $m\in\Mon$,
and $I(p)$ to $I(p[\Mon])$, if $p\in\Fun{\Mon,\Mon}$.
}
\end{equation}
The map $A\to I(A)$ is clearly a closure operation which 
induces the order topology $\tau_o$. A map between 
two p.o. monoids, each endowed with its own order 
topology, is continuous if and only if it preserves order.
The composition operation is thus a continuous mapping 
of $\Mon\times\Mon$ to $\Mon$ and $(\Mon,\tau_o)$
a topological p.o. monoid.

Further classes of ideals are of interest. $\Id$ is a radical 
ideal if $m\in\Id$ whenever $m^j\in\Id$ for some $j\in\N$.
We write 
$\sqrt{\Id}
	=
\{m\in\Mon:m^j\in\Id\text{ for some }j\in\N\}$.
$\Id$ is 
radical if and only if $\Mon/\Id$ is normal. 
We also define a $D$-ideal to be an $o$-ideal 
$\Id$ such that $\Id\cap I(m)$ 
admits a greatest element for each $m\in\Mon$.

\begin{lemma}
\label{lemma $o$-ideal}
$\Id\subseteq\Mon$ is a $D$-ideal if and only if 
$\ \Id=p[\Mon]$ for some order preserving 
$p\in\Fun{\Mon,\Mon}$ satisfying
\begin{subequations}
\label{proj}
\begin{equation}
\label{proj contractive}
m\ge p(m)
\qquad
m\in\Mon
\qtext{and}
\end{equation}
\begin{equation}
\label{proj idempotent}
p(m)
	\ge 
n
\qtext{implies}
n
	=
p(n)
\qquad
m,n\in\Mon.
\end{equation}
\end{subequations}
\end{lemma}

\begin{proof}
If $p$ satisfies \eqref{proj idempotent} its range 
is necessarily an $o$-ideal; if, in addition, $p$ 
preserves order and satisfies \eqref{proj contractive}, 
then $p(n)\le m$ implies $p(n)\le p(m)\le m$ so 
that its range is a $D$-ideal. Conversely, if 
$\Id$ is a $D$-ideal and if $p(m)$ is the greatest 
element in the set $\Id\cap I(m)$, then the map $p$ 
clearly preserves order and satisfies \eqref{proj contractive}. 
Moreover, if $n\le p(m)$ then $n\in\Id$ so that 
$\Id\cap I(n)
	=
I(n)$,
i.e. $n=p(n)$ so that $I(p)=p[\Mon]$. Eventually,
$
\Id
	=
\bigcup_{m\in\Mon}\Id\cap I(m)
	=
\bigcup_{m\in\Mon}I\big(p(m)\big)
	=
I(p)
$.
\end{proof}

By analogy with the theory of vector lattices, a map
with the properties of Lemma \ref{lemma $o$-ideal}
is called an $o$-projection. The family $\Proj$ of 
$o$-projections forms a p.o., idempotent monoid 
if endowed with composition. 
The corresponding algebraic notion is that of a 
(monoid) projection i.e. an order preserving map 
$q\in\Fun{\Mon,\Mon}$
which satisfies \eqref{proj contractive} and
\begin{equation}
\label{proj M} \tag{\ref{proj}c}
mq(n)
	\le
q(mn)
\qquad
m,n\in\Mon.
\end{equation}
Denoting by $\SProj$ the family of projections we have 
$\Proj\subseteq\SProj$%
\footnote{
In fact, if $p\in\Proj$ then \eqref{proj idempotent}
and $mp(n)\le p(n)$ imply $mp(n)=p(mp(n))$ while 
\eqref{proj contractive} implies $p(mp(n))\le p(mn)$.
}
 while the converse holds if and 
only if $\Mon$ is idempotent. More interestingly, each 
$m\in\Mon$ corresponds with a projection $m^*$ on 
$\Proj$ via the identity
\begin{equation}
\label{iso}
\big(m^*(p)\big)(n)
	=
p(nm)
\qquad
p\in\Proj,\ 
n\in\Mon.
\end{equation}

The main example of a projection is the translate 
$T_m$ by $m$, defined as $T_m(n)=mn$.

\section{$\kappa$-ideals}
\label{sec domains}

The following is a generalisation of the classical 
notion of a prime ideal (due to Tarski 
\cite[Definition 4.1]{tarski_45}) and often used
in set theory. e adapt the definition to p.o. 
monoids.

\begin{definition}
\label{def kappa-ideal}
Let $\kappa\ge2$ be a cardinal. A $\kappa$-ideal
in $\Mon$ is an $o$-ideal $\Id$ with the property 
that every mutually $\Id$-disjoint subset of $\Mon$ 
is a $\kappa$-set. A $2$-ideal is referred to as a
prime ideal%
\footnote{
In the terminology of \cite{jech_book}, a $\kappa$-ideal 
is a $\kappa$-saturated $o$-ideal. The restriction to 
$o$-ideals is only for terminological convenience. 
}.
\end{definition}

Notice that if $\Id$ is a $\kappa$-ideal and $\Id'$
is an $o$-ideal contained in $\Id$, then $\Id'$ need
not be a $\kappa$-ideal.

Some properties of prime ideals carry over unchanged 
from the theory of rings. The family of prime ideals 
contains $\emp$ and is closed with respect to arbitrary 
unions. Their complements form thus a base for a 
topology, $\tau_p$. If $q\in\SProj$ and $\Id$ is a 
prime ideal, then $q^{-1}(\Id)$ is clearly an $o$-ideal. 
Moreover, if $m,n\in\Mon$ are such that $q(mn)\in\Id$, 
then $q(m)q(n)\in\Id$, by \eqref{proj M} and 
\eqref{proj contractive}. We must then have either 
$q(n)\in\Id$ or $q(m)\in\Id$. Thus $q^{-1}(\Id)$ is
a prime ideal when $\Id$ is so.

\begin{lemma}
\label{lemma prime}
(i)
$(\Mon,\tau_p)$ is a topological p.o. monoid
in which each $q\in\SProj$ is continuous,
(ii)
if $\{0\}$ is prime, the annihilator $M^\perp$ 
of any $M\subseteq\Mon$ is $\tau_p$-closed and
(iii)
each radical $o$-ideal is the intersection of all
prime ideals containing it (and is thus closed).
\end{lemma}

\begin{proof}
Radical, $o$-ideals are closed with respect to 
arbitrary unions. If $\Id$ is such an ideal and
if $a\notin\Id$, by Zorn lemma we can form a 
maximal radical $o$-ideal $\Id_a$ which includes 
$\Id$ but not $a$. Suppose that $m,n\in\Mon$
are such that $mn\in\Id_a$. If $n,m\in\Id_a^c$
then the radical $o$-ideals $\Id_a\cup\sqrt{I(m)}$ 
and $\Id_a\cup\sqrt{I(n)}$ both contain $a$, 
by maximality. There exist then $j,k\in\N$ such 
that $a^j\le m$ and $a^k\le n$ and thus 
$a^{j+k}\le mn\in\Id_a$. But this implies 
$a\in\Id_a$, a contradiction. Thus $\Id_a$ is 
prime and $\Id=\bigcap_{a\notin\Id}\Id_a$.
\end{proof}

The well-ordering principle permits the following 
definition:

\begin{definition}
\label{def K(I)}
Given a subset $ M\subseteq\Mon$ and an ideal 
$\Id$ in $\Mon$ we define $\kappa( M,\Id)$ to
be the least cardinal number $>\card{ M_0}$ for 
any mutually $\Id$-disjoint subset 
$M_0\subseteq M$. We write 
$\kappa(M,\{0\})=\kappa(M)$.
\end{definition}

By definition, every $o$-ideal is a $\kappa(\Mon,\Id)$%
-ideal. In applications, we shall mainly encounter the 
case
$\kappa( M,\Id)
	\le
\aleph_1$.
In general, computing $\kappa( M,\Id)$ may not 
be easy. We provide some explicit examples.

\begin{example}
\label{ex separable}
Let $\Mon$ be the monoid of real valued, non negative, 
lower semicontinuous functions on some topological 
space $X$ with binary operation $fg=f\wedge g$.
If $X$ is separable, then $\kappa(\Mon)\le\aleph_1$;
if $X$ is compact and totally disconnected 
then $\kappa(\Mon)\le\aleph_0$.
\end{example}

\begin{example}
\label{ex monoid}
Consider a commutative monoid $\Mon$ with its 
natural order. Then, ideals and $o$-ideals 
coincide. Let $m_1,\ldots,m_N$ be distinct, irreducible 
elements of $\Mon$ and let $m_0=\prod_{i=1}^Nm_i$. 
The ideal $I(m_0)$ is clearly seen to be a radical $o$-%
ideal. Consider $M\subseteq\Mon$ to be mutually 
$I(m_0)$-disjoint. Then, for each $h\in M$ there must 
be an integer $1\le i\le N$ such that $m_i$ 
does not divide $h$. At the same time, since 
$hf\in I(m_0)$ when $h,f\in M$, for each $i$ there
is at most one element in $M$ which is not
divided by $m_i$. It follows that $\card{ M}\le N$.
\end{example}

\begin{example}
\label{ex lattice}
Let $L$ be an $AL$-space (\cite[p. 193]{aliprantis_burkinshaw}). 
Fix $x_0\in L_+$ and consider $\Mon=\{x\in L:0\le x\le x_0\}$ 
endowed with the binary operation $\wedge$. 
If $x_1,\ldots,x_n\in\Mon$ are mutually disjoint,
then
\begin{equation}
\label{lat ineq}
\norm{x_0}
	\ge 
\norm{x_1\vee\ldots\vee x_n}
	=
\norm{x_1}+\ldots+\norm{x_n}.
\end{equation}
This implies that $\kappa(\Mon)\le\aleph_1$.
\end{example}

Erd\H{o}s and Tarski \cite{erdos_tarski} proved that the 
general conjecture that $\kappa( M,\Id)$ may be {\it any} 
cardinal number is false. In the next Lemma we adapt their 
result to the present setting.

\begin{lemma}[Erd\H{o}s and Tarski]
\label{lemma erdos tarski}
Let $\Id$ be a radical $o$-ideal in $\Mon$ and fix
$M\subseteq\Mon$. Then
$\kappa( M,\Id)
	=
\kappa( M/\Id)$ 
and $\kappa( M,\Id)$ cannot be a singular limit 
cardinal nor $\aleph_0$.
\end{lemma}

\begin{proof}
The inequality
$\kappa( M/\Id)
	\le
\kappa( M,\Id)$
is clear. 
Let $m,n\in\Id^c$ be distinct,
$\Id$-disjoint  elements such that $m/\Id=n/\Id$. Then 
necessarily
\begin{equation*}
0/\Id
	=
(mn/\Id)
	=
(m/\Id)(n/\Id)
	=
(m/\Id)(m/\Id)
	=
(m^2/\Id)
\end{equation*}
i.e.
$m^2\in\Id$ which is impossible if $\Id$ is 
radical.

The second claim follows from \cite[Theorem 1]{erdos_tarski} 
once we prove that two elements $x,z\in\Mon/\Id$ are mutually
disjoint if and only if there is no $y\in\Mon/\Id$ such that 
$y\le x$ and $y\le z$ other than $0/\Id$. One implication 
follows from the fact that $\Mon/\Id$ is a p.o. monoid so 
that $xz\le x$ and $xz\le z$. Conversely, if $h/\Id\le m/\Id$ 
and $h/\Id\le n/\Id$ for some $m,n,h\in\Mon$ with $mn\in\Id$, 
we conclude
\begin{equation}
(h^2/\Id)
	=
(h/\Id)(h/\Id)
	\le
(n/\Id)(m/\Id)
	=
0/\Id,
\end{equation}
i.e. that $h\in\Id$ since $\Id$ is radical. This shows that 
when $\Id$ is a radical, $o$-ideal a subset of $\Mon/\Id$ 
is mutually disjoint in our definition if and only if it is so 
in the sense of \cite[p. 316]{erdos_tarski}. 
\end{proof}

The simple properties of prime ideals listed in Lemma
\ref{lemma prime} need not be true in the case of
$\kappa$-ideals.

\begin{lemma}
\label{lemma k}
If $\kappa$ is a regular cardinal number then the
intersection of a $\kappa$-family of $\kappa$-ideals
is a $\kappa$-ideal and, if $\kappa=\aleph_0$, so is
their union.
\end{lemma}

\begin{proof}
First of all,
$\Id_0=\bigcap_\alpha\Id_\alpha$ 
and 
$\Id_1=\bigcup_\alpha\Id_\alpha$
are $o$-ideals if $\Id_\alpha$ is so for each index 
$\alpha$ in the $\kappa$-set $\mathfrak A$. Choose 
$ M\subseteq\Mon$ to be mutually $\Id_0$-disjoint. 
Then, $ M\subseteq\Id_0^c$. Write
$ M_\alpha
	=
 M\cap\Id_\alpha^c$. 
Of course, $ M_\alpha$ is mutually $\Id_\alpha$-disjoint.
But then, given that 
$ M
	=
\bigcup_\alpha M_\alpha$,
\begin{equation}
\card{ M}
	\le
\sum_{\alpha\in\mathfrak A}\card{ M_\alpha}
	\le
\card{\mathfrak A}\cdot\kappa
	\le
\kappa.
\end{equation}
However, if $\card{ M}=\kappa$ then $\kappa$ is a 
singular, limit cardinal a contradiction. This proves 
that $\Id_0$ is a $\kappa$-ideal. 

Concerning union, assume that $\mathfrak A$ is a 
finite set and choose $ M\subseteq\Mon$ to be mutually
$\Id_1$-disjoint. Suppose that $\card M\ge\aleph_0$. 
By passing to a subset, we can assume with no loss of 
generality that $\card M=\aleph_0$. Then, since 
$mn\in\Id_\alpha$ for some $\alpha\in\mathfrak A$,
we conclude from a well known result of Ramsey 
\cite[Theorem A]{ramsey} (see also Erd\H{o}s and Rado
\cite[Theorem 1]{erdos_rado}) that there exists
$\alpha_0\in\mathfrak A$ and a subset $ M_0\subseteq M$
such that $\card{ M_0}=\aleph_0$ and that $ M_0$
is $\Id_{\alpha_0}$-disjoint, which contrasts with the
assumption that each $\Id_\alpha$ is a $\kappa$-ideal. 
Thus necessarily $\card M<\aleph_0$.
\end{proof}

\section{The projection problem}
\label{sec projection}

Each $D$-ideal is the range of some $o$-projection. 
In general, the question whether a projection maps 
a given subset of $\Mon$ into some ideal is not 
trivial and we refer to it as the projection problem. 
More formally,

\begin{definition}[Projection Problem]
Given $ M\subseteq\Mon$, $Q\subseteq\SProj$ and 
an ideal $\Id$ in $\Mon$ is there $q\in Q$ such 
that $q[ M]\subseteq\Id$?
\end{definition}

This problem has an interesting structure in the 
case in which $\Id$ is a radical, $o$-ideal. Given
the preceding remarks, there is no loss of generality
in setting $\Id=\{0\}$ and assuming that $\Mon$ 
is normal. The following result establishes the
existence of sufficient subsets of bounded cardinality.

\begin{theorem}
\label{th main}
Assume that $\Mon$ is normal and let
$Q\subseteq\SProj$. Each $ M\subseteq\Mon$ admits 
a $\kappa\big(Q[ M]\big)$-subset $M_0\subseteq M$ 
which is sufficient for $Q$, i.e. such that
\begin{equation}
\label{main}
q[ M]=\{0\}
\qtext{if and only if}
q[ M_0]=\{0\}
\qquad
q\in Q.
\end{equation}
Moreover, $Q[M]\cap Q[ M_0]^\perp\subseteq\{0\}$.
\end{theorem}

\begin{proof}
Form the class of all subsets of $Q[M]$ which 
are mutually disjoint and choose, by virtue of 
Zorn lemma, a maximal set 
$\{q_\alpha(m_\alpha):\alpha\in\mathfrak A\}$
in this class. Write  
$M_0
	=
\{m_\alpha:\alpha\in\mathfrak A\}$.
Then,
$\card{M_0}
	\le
\card{\mathfrak A}
	<
\kappa\big(Q[M]\big)$.
To prove the last claim first, let $q(n)\in Q[M]$
be such that $q(n)q_\alpha(m_\alpha)=0$ for all 
$\alpha\in\mathfrak A$. Given that $\Mon$ is normal, 
the collection
$\{q_\alpha(m_\alpha):\alpha\in\mathfrak A\}
	\cup
\{q(n)\}$
contains 
$\{q_\alpha(m_\alpha):\alpha\in\mathfrak A\}$
properly so that $q(n)\ne0$ would contradict
maximality. Let $q_0\in Q$ satisfy
$q_0[M_0]=\{0\}$. By \eqref{proj M},
\begin{align}
q_0(m)q_\alpha(m_\alpha)
	\le
q_0(mm_\alpha)
	\le
q_0(m_\alpha)
	=
0
\qquad
m\in M,\ \alpha\in\mathfrak A.
\end{align}
Then 
$q_0(m)
	\in
\{q_\alpha(m_\alpha):\alpha\in\mathfrak A\}^\perp$
so that necessarily $q_0[M]=\{0\}$. 
\end{proof}

A number of implications of Theorem \ref{th main} may 
be obtained right away. 

\begin{corollary}
\label{cor main}
(i)
If $\Mon$ is normal, then each set $M\subseteq\Mon$
admits a $\kappa(I(M))$-subset $M_0$ such that
\begin{equation}
\label{orth}
M_0^\perp
	=
M^\perp.
\end{equation}
In particular, $M$ and $M_0$ have the same $\gez$-upper
bounds.
(ii)
Each $P\subseteq\Proj$ admits a $\kappa(I(P))$-subset
$P_0$ such that
\begin{equation}
\label{zeros}
\bigcap_{p\in P}p^{-1}(0)
=
\bigcap_{p\in P_0}p^{-1}(0).
\end{equation}
\end{corollary}

\begin{proof}
The first claim follows from Theorem \ref{th main} 
upon letting $Q$ consist of all translates by some 
$m\in\Mon$ and noting that, in this case, 
$Q[M]\subseteq I(M)$. By definition \eqref{quotient} 
an element $n\in\Mon$ is an $\gez$-upper bound 
for $M_0$ if and only if $\{n\}^\perp\subseteq M_0^\perp$ 
i.e, by \eqref{orth}, if and only if it is a $\gez$-upper 
bound for $M$. To prove \tiref{ii} recall that each 
$m\in\Mon$ acts as an $o$-projection $m^*$ on 
the normal p.o. monoid $\Proj$. Then, 
$m\in\bigcap_{p\in P}p^{-1}(0)$ 
is equivalent to $m^*[P]=\{0\}$ and 
$\{m^*p:m\in\Mon,\ p\in P\}\subseteq I(P)$.
\end{proof}

For a set $M$ of positive elements in a vector
lattice $X$ the condition $\kappa(I(M))\le\aleph_1$ 
implies then that $M\subseteq\{m\}^{\perp\perp}$ 
for some $m\in X_+$.

%

The projection problem may also be studied locally,
by looking at the sets
\begin{equation}
\label{Qa(f;I)}
\SProj(m)
	=
\big\{q\in \SProj:q(m)\ne0\big\}
\qquad
m\in M
\end{equation}
which are open in the order topology of $\SProj$. 
Theorem \ref{th main} translates into a compactness
statement: if $\{\SProj(m):m\in M\}$ covers $Q$ 
it admits then a $\kappa(Q[M])$-subcover. But 
then, if $Q$ is a ``large'' set and if 
$Q
	\subseteq
\bigcup_{m\in M}\SProj(m)$, 
then one of such sets must be ``large'' as well. 
This version of the pigeonhole principle admits 
a rigorous formulation.

\begin{theorem}
\label{th regular}
Let $\Mon$, $ M$ and $Q$ be as in Theorem 
\ref{th main}. Let the cardinal $\na$ be the 
greatest of $\kappa(Q[ M])$ and $\aleph_0$. 
If
\begin{equation}
\label{iineq}
\card Q
	\ge
\na
	>
\card{Q\cap\SProj(m)}
\qquad
m\in M
\end{equation}
then the set $Q_0=\{q\in Q:q[ M]=\{0\}\}$
has the same cardinality as $Q$.
\end{theorem}

\begin{proof}
If $\na$ is as in the statement it is then necessarily
a regular cardinal number. Let $ M_0\subseteq M$ be
a $\kappa(Q[M])$-subset sufficient for $Q$. Define
$Q_1
	=
Q\cap\bigcup_{m\in M_0}\SProj(m)$.
Clearly, $Q_0=Q\setminus Q_1$. 
Moreover, by basic cardinal arithmetic
\begin{equation}
\label{comp}
\card{Q_1}
	\le
\sum_{m\in M_0}\card{Q\cap\SProj(m)}
	\le
\na\cdot\na
	=
\na.
\end{equation}
However, since $\na$ is regular by 
\cite[Lemma 3.10]{jech_book} we must have
$\card{Q_1}<\na$ and
$\card{Q_0}
	=
\card{Q}$.
\end{proof}

Theorem \ref{th regular} applies, e.g., when 
$\card{Q}
	=
\aleph_1
	\ge
\kappa(\Mon)$ 
(in fact, every successive cardinal is regular, see 
\cite[Corollary 5.3]{jech_book}). 

Eventually, Theorem \ref{th main} translates into a
result on partitions.

\begin{theorem}
\label{th partition}
Let $\Mon$, $ M$ and $Q$ be as in Theorem \ref{th main}
and let $\na$ be a regular cardinal number 
$\ge\kappa(Q[ M])$. Assume that $\Mon$ 
decomposes as
\begin{equation}
\Mon
	=
\{0\}\cup\bigcup_{\alpha\in\mathfrak A}\Mon_\alpha
\end{equation}
in which $\mathfrak A$ is a $\na$-set and 
$0\notin\Mon_\alpha$ for each $\alpha\in\mathfrak A$. 
Then, $Q$ admits the decomposition
\begin{equation}
\label{decomp}
Q
	=
Q_0\cup\bigcup_{\beta\in\B}Q_\beta
\end{equation}
in which $\B$ is a $\na$-set, $Q_0[ M]=\{0\}$ 
while for each $\beta\in\B$ there exist 
$\alpha_\beta\in\mathfrak A$ and $m_\beta\in M$ 
such that 
\begin{equation}
Q_\beta[m_\beta]\subseteq\Mon_{\alpha_\beta}.
\end{equation}
\end{theorem}

\begin{proof}
Define
$Q_0
=
\{q\in Q:q[ M]=\{0\}\}$.
According to Theorem \ref{th main} we can extract 
a $\na$-set $ M_0\subseteq M$ which is sufficient for 
$Q$. Then, $q\notin Q_0$ if and only if there exists 
a pair $(\alpha,m)\in\mathfrak A\times M_0$ such 
that $q(m)\in\Mon_\alpha$. Let $\B$ be the set of all
such pairs and write each $\beta\in\B$ in the form 
$(\alpha_\beta,m_\beta)$. Given that $\mathfrak A$ 
and $ M_0$ are $\na$-sets and that $\na$ is regular
$
\card{\B}
	\le
\sum_{\alpha\in\mathfrak A}\card{ M_0}
	<
\na
$.
The 
decomposition \eqref{decomp} becomes obvious 
upon defining
$Q_\beta
	=
\{q\in Q:q(m_\beta)\in\Mon_{\alpha_\beta}\}$
for all
$\beta\in\B$.
\end{proof}

\section{Functions of Bounded Variation.}
\label{sec FV}
We provide some more explicit applications. If we
define the group operation on $\R_+$ by $\wedge$
we obtain that $\Fun{X,\R_+}$ is an idempotent,
product p.o. monoid.

\begin{corollary}
Let $X$ be a non empty set and let $\F$ be an 
ideal in the p.o. product monoid $\Fun{X,\R_+}$. 
Assume that $\kappa(\F)\le\aleph_1$. Then there 
exist $f_1,f_2,\ldots\in\F$ such that
\begin{equation}
\sup_nf_n(x)=0
\qiff
\sup_{f\in\F}f(x)=0.
\end{equation}
\end{corollary}

\begin{proof}
Upon replacing $f$ with $ f/(1+ f)$ there 
is no loss of generality in assuming that 
$\F\subseteq\Fun{X,[0,1]}$. The map 
$q_A(f)=f\set A$ is an $o$-projection for 
each $\emp\ne A\subseteq X$. 
Let $Q$ be the corresponding family. We have 
$\kappa(Q[\F])
	\le
\kappa(\F)
	\le
\aleph_1$.
The claim then follows from Theorem \ref{th main}.
\end{proof}

More interesting conclusions may be obtained with
additional structure. 

\begin{definition}
An increasing function $f\in\Fun{\Mon,\R}$ is of 
finite variation, in symbols $f\in\mathscr V(\Mon)$, 
if
\begin{equation}
\label{norm}
f(0)=0
\qand
\norm f_{_{\VM}}
=
\sup_{M\subseteq\Mon}
\sum_{m\in M}f(m)
	<
+\infty
\end{equation}
where $M$ ranges over all finite, mutually
$f^{-1}(0)$-disjoint subsets of $\Mon$. 
\end{definition}

$\VM$ is a p.o. sub semigroup in $\Fun{\Mon,\R_+}$ 
and for each $f\in\VM$ the set $f^{-1}(0)$ is an 
$\aleph_1$-ideal. Thus, each set $M\subseteq\Mon$ 
contains a countable subset $m_1,m_2,\ldots\in M$ 
such that for each
$n\in\Mon$
\begin{equation}
\sup_{m\in M}f(nm)=0
\qiff
\sup_{i\in\N}f(nm_i)=0.
\end{equation}

\begin{example}
\label{ex capacity}
Let $\Mon$ be an idempotent, p.o. monoid and let the 
function $f\in\Fun{\Mon,\R}$ satisfy $f(0)=0$ and
\begin{equation}
\label{capacity}
f(m)
	\ge
\sum_{\emp\ne M_0\subseteq M}(-1)^{1+\card{M_0}}
f\Big(\prod_{n\in M_0} n\Big)
\qquad
m\in\Mon,\ 
M\subseteq I(m)\text{ finite}.
\end{equation}
If $\Mon$ is a $\pi$ system of sets then \eqref{capacity} 
corresponds to the definition of a supermodular capacity 
given by Choquet \cite[p. 171]{choquet}. If $M\subseteq\Mon$ 
is mutually $f^{-1}(0)$-disjoint, then
$f(1)
	\ge
\sum_{m\in M}f(m)$ so that $f\in\VM$.
\end{example}

\begin{theorem}
\label{th VM}
Let $V\subseteq\VM$ be such that $f\set{I(m)}\in V$
whenever $f\in V$ and $m\in\Mon$. The following 
properties are mutually equivalent:
\begin{enumerate}[(a)]
\item
\label{IU}
for every $\emp\ne U\subseteq V$ the set 
$\bigcap_{f\in U} f^{-1}(0)$
is an $\aleph_1$-ideal in $\Mon$;
\item\label{k(V)}
$\kappa(V)\le\aleph_1$;
\item
\label{wac}
every $\emp\ne U\subseteq V$ admits some
$f_{_U}\in\VM$ such that
$f_{_U}^{-1}(0)
	=
\bigcap_{f\in U}f^{-1}(0)$.
\end{enumerate}
\end{theorem}

\begin{proof}
\imply{IU}{k(V)}
If $\emp\ne U\subseteq V$ is mutually disjoint then
any two elements $f,g\in U$ have disjoint support. 
For each $g\in U$ we can thus select $m_g\in\Mon$ 
such that $g(m_g)>0$. The family $\{m_g:g\in U\}$ 
is mutually $\bigcap_{f\in U} f^{-1}(0)$-disjoint 
and has the same cardinality as $U$. If 
$\bigcap_{f\in U} f^{-1}(0)$ is an 
$\aleph_1$-ideal, then necessarily $U$ is countable.
This proves that $\kappa(V)\le\aleph_1$. 
\imply{k(V)}{wac}
For each $m\in\Mon$ consider the projection 
$q_m\in\SProj(\VM)$ implicitly defined by
$q_mf=f\set{I(m)}$ and apply Theorem \ref{th main} 
with $Q=\{q_m:m\in\Mon\}$ and $M=U$. Observe 
that $q_mf=0$ if and only if $f(m)=0$. By assumption, 
$Q[U]\subseteq V$ so that 
$\kappa(Q[U])
	\le
\aleph_1$.
We obtain a countable subset $U_0\subseteq U$ such 
that for any $m\in \Mon$, $q_m[U_0]=0$ if and only if 
$q_m[U]=0$. If $f_1,f_2,\ldots$ is an enumeration of 
$U_0$ we can define $f_{_U}\in\Fun{\Mon,\R}$ by letting
\begin{equation}
\label{fU}
f_{_U}(m)
	=
\sum_j2^{-j}f_j(m)/\big(1+\norm{f_j}_{_{\VM}}\big)
\qquad
m\in\Mon.
\end{equation}
Notice that $f_{_U}\in\VM$ and that $f_{_U}(m)=0$ 
is equivalent to $\sup_{f\in U}f(m)=0$.
\imply{wac}{IU}
If $\emp\ne U\subseteq V$ and if 
$\bigcap_{f\in U}f^{-1}(0)
	=
f_{_U}^{-1}(0)$
for some $f_{_U}\in\VM$, then necessarily 
$\bigcap_{f\in U}f^{-1}(0)$ is certainly an 
$\aleph_1$-ideal by the finite variation 
property.
\end{proof}

Two final comments. First, the proof shows that if for 
given $U\subseteq V$ there exists $f_{_U}\in\VM$ 
satisfying \iref{wac}  then it can always be taken to 
be of the form \eqref{fU}. This is important since if 
$U$ consists e.g. of supermodular capacities, so will 
be $f_{_U}$. Second, Theorem \ref{th VM} may be 
useful to prove the existence of some $f\in\VM$ 
which is positive on all $m\in\Mon$, $m\ne0$. This 
is an important and challenging problem (that goes 
much beyond the scope of this work). In the case in 
which $\Mon$ is a Boolean algebra and $V$ consists 
of additive probabilities on $\Mon$, this is just the 
problem raised by Maharam in \cite{maharam} (see 
\cite{talagrand} for a partial, negative answer). The 
problem makes however sense also in other classes 
of functions of finite variation e.g. not additive set 
functions on a given algebra of sets. In applying
condition \iref{k(V)} of Theorem \ref{th VM} one 
should recall that the disjointness condition formulated
for $\VM$ need not coincide with the corresponding
condition e.g. for additive set functions. 

\BIB{acm}

\end{document}